\documentclass[reqno]{amsart}

\usepackage{amsmath}
\usepackage{amsfonts}
\usepackage{amssymb,enumerate}
\usepackage{amsthm}
\usepackage[all]{xy}
\usepackage{rotating}
\usepackage{hyperref}

\theoremstyle{plain}
\newtheorem{lem}{Lemma}[section]
\newtheorem{cor}[lem]{Corollary}

\newtheorem{thm}[lem]{Theorem}

\theoremstyle{definition}
\newtheorem{defn}[lem]{Definition}
\newtheorem{ex}[lem]{Example}

\newtheorem{rmk}[lem]{Remark}

\newtheorem{fact}[lem]{Fact}
\newtheorem{para}[lem]{}

\newcommand{\cat}[1]{\mathcal{#1}}

\newcommand{\catd}{\cat{D}}

\newcommand{\cata}{\cat{A}}

\newcommand{\catb}{\cat{B}}

\newcommand{\rank}{\operatorname{rank}}

\newcommand{\ann}{\operatorname{Ann}}

\newcommand{\HH}{\operatorname{H}}
\newcommand{\Hom}{\operatorname{Hom}}	

\newcommand{\spec}{\operatorname{Spec}}

\newcommand{\shift}{\mathsf{\Sigma}}

\newcommand{\cone}{\operatorname{Cone}}

\newcommand{\ideal}[1]{\mathfrak{#1}}
\newcommand{\m}{\ideal{m}}

\newcommand{\p}{\ideal{p}}

\newcommand{\ol}{\overline}

\newcommand{\supp}{\operatorname{Supp}}

\newcommand{\bbz}{\mathbb{Z}}

\newcommand{\xra}{\xrightarrow}

\newcommand{\into}{\hookrightarrow}

\renewcommand{\geq}{\geqslant}
\renewcommand{\leq}{\leqslant}

\newcommand{\Ext}[4][R]{\operatorname{Ext}_{#1}^{#2}(#3,#4)}	
\newcommand{\Rhom}[3][R]{\mathbf{R}\!\operatorname{Hom}_{#1}(#2,#3)}	
\newcommand{\Lotimes}[3][R]{#2\otimes^{\mathbf{L}}_{#1}#3}
\newcommand{\Otimes}[3][R]{#2\otimes_{#1}#3}
\renewcommand{\Hom}[3][R]{\operatorname{Hom}_{#1}(#2,#3)}	
\newcommand{\Tor}[4][R]{\operatorname{Tor}^{#1}_{#2}(#3,#4)}

\newcommand{\ssm}{\smallsetminus}

\newcommand{\sd}{\operatorname{SD}}

\numberwithin{equation}{lem}

\begin{document}

\bibliographystyle{amsplain}

\author{Sean Sather-Wagstaff}

\address{NDSU Department of Mathematics \# 2750,
PO Box 6050,
Fargo, ND 58108-6050
USA}

\email{sean.sather-wagstaff@ndsu.edu}

\urladdr{http://www.ndsu.edu/pubweb/\~{}ssatherw/}

\title{Reflexivity and connectedness}

\date{\today}

\keywords{Auslander classes, Bass classes, derived reflexive complexes, disconnected prime spectrum, 
Foxby classes, 
semidualizing complexes, semidualizing modules, totally reflexive modules}
\subjclass[2010]{13D02,13D07, 13D09, 13G05}

\begin{abstract}
Given a finitely generated module
over a commutative noetherian ring that satisfies certain reflexivity conditions,
we show how  failure of the semidualizing property for the module manifests 
in a disconnection of the prime spectrum of the ring.
\end{abstract}

\maketitle

\section{Introduction} \label{sec0}

Throughout this paper, $R$ is a non-zero commutative noetherian ring with identity and all $R$-modules are unital.

An $R$-module $A$ is a \emph{semidualizing} if
the natural homothety map $\chi^R_A\colon R\to\Hom AA$ is an isomorphism
and $\Ext iAA=0$ for all $i\geq 1$.
These gadgets, and their cousins the semidualizing complexes, are useful for studying dualities.
For instance, their applications include Grothendieck's local duality~\cite{hartshorne:rad, hartshorne:lc},
progress by Avramov-Foxby~\cite{avramov:rhafgd} and Sather-Wagstaff~\cite{sather:cidfc}
on the composition question for local ring homomorphisms of finite G-dimension,
and progress by Sather-Wagstaff~\cite{sather:bnsc} on Huneke's question on the behavior of
Bass numbers of local rings.

The starting point for the current paper is the following straightforward result:

\begin{fact}\label{fact121124a}
For a finitely generated $R$-module $A$,  the next conditions are equivalent:
\begin{enumerate}[(i)]
\item \label{fact121124a1}
$A$ is a \emph{semidualizing} $R$-module, 
\item \label{fact121124a2}
$R$ is a \emph{totally $A$-reflexive} $R$-module, i.e., the natural biduality map $\delta^R_A\colon R\to\Hom{\Hom RA}A$ is an isomorphism
and $\Ext i{\Hom RA}A=0=\Ext iR{A}$ for all $i\geq 1$, and
\item \label{fact121124a3}
$A$ is totally $A$-reflexive and $\ann_R(A)=0$.
\end{enumerate}
\end{fact}

It is straightforward to show that the annihilator condition in item~\eqref{fact121124a3} is necessary:
if $A$ is totally $A$-reflexive, then $A$ need not be semidualizing.
For instance, if $A=0$, then $A$ is totally $A$-reflexive but is not semidualizing.
A slightly less trivial example is the following:

\begin{ex}\label{ex121124a}
Let $R_1,R_2$ be non-zero commutative noetherian rings with identity, and 
set $R=R_1\times R_2$. Then the $R$-module $A=R_1\times 0$ is totally $A$-reflexive
but is not semidualizing. Moreover, given any semidualizing $R_1$-module $A_1$,
the $R$-module $A=A_1\times 0$ is totally $A$-reflexive
but is not semidualizing.
\end{ex}

The point of this paper is to show that this is the only way for this to occur. 
Specifically, we prove the following in~\ref{proof121130a}:

\begin{thm}\label{thm121124a}
Let $A$ be a non-zero finitely generated $R$-module that is totally $A$-reflexive and not semidualizing.
Then there are  commutative noetherian rings $R_1,R_2\neq 0$ with identity
such that $R\cong R_1\times R_2$, and there is a semidualizing $R_1$-module $A_1$
such that $A\cong A_1\times 0$. In particular, $\spec(R)$ is disconnected.
\end{thm}

If $R$ is local or a domain, then $\spec(R)$ is connected. Hence, if $A$ is non-zero and totally $A$-reflexive,
then $A$ must be semidualizing. The local version of this is actually a key point of the proof of Theorem~\ref{thm121124a},
which is contained in Theorem~\ref{thm121124e} below. The version for domains is documented in Corollary~\ref{cor121130a}.
Note that our results also give other conditions on $A$ that imply that $\spec(R)$ is disconnected or that $A$ is semidualizing.
These conditions are akin to those studied in~\cite{avramov:rrc1,frankild:rbsc}.
\section{Background}\label{sec121130a}

We begin this section with some background information.

\begin{para}\label{para121124a}
We work in the derived category $\catd(R)$ where each $R$-complex $X$ is indexed homologically:
$X=\cdots\to X_i\xra{\partial^X_i}X_{i-1}\to\cdots$. 
An $R$-complex $X$ is \emph{homologically bounded}
if  $\HH_i(X)=0$ for all but finitely many $i$.
The complex $X$ is \emph{homologically finite}
if it is homologically bounded and $\HH_i(X)$ is finitely generated for all $i$.
The $i$th \emph{suspension} of $X$ is $\shift^iX$.
Isomorphisms in $\catd(R)$ are identified with the symbol $\simeq$.
Two $R$-complexes $X$ and $Y$ are \emph{shift-isomorphic}, written $X\sim Y$, if there is an integer $i$ such that
$X\simeq\shift^iY$. 
The \emph{large support} of $X$ is
$\supp_R(X):=\{\p\in\spec(R)\mid X_{\p}\not\simeq 0\}$.
Given two $R$-complexes $X$ and $Y$, the \emph{right derived Hom complex} and \emph{left derived tensor product complex}
of $X$ and $Y$ are denoted $\Rhom XY$ and $\Lotimes XY$,
and $\Ext iXY:=\HH^i(\Rhom XY)$.

If $(R,\m,k)$ is local,
then the \emph{Bass series} and \emph{Poincar\'e series} of a homologically finite $R$-complex $X$ are the formal Laurent series
\begin{align*}
I_R^X(t)&=\sum_{i\in\bbz}\rank_k(\HH^i(\Rhom kX))t^i\\
P^R_X(t)&=\sum_{i\in\bbz}\rank_k(\HH^i(\Lotimes kX))t^i.
\end{align*}
\end{para}

Semidualizing complexes and the various classes that they define  originate in work of
Auslander-Bridger~\cite{auslander:adgeteac,auslander:smt},
Avramov-Foxby~\cite{avramov:rhafgd},
Christensen~\cite{christensen:scatac},
Enochs-Jenda-Xu~\cite{enochs:fdgipm},
Foxby~\cite{foxby:gmarm,foxby:gdcmr}, 
Golod~\cite{golod:gdagpi}, 
Vasconcelos~\cite{vasconcelos:dtmc}, and
Yassemi~\cite{yassemi:gd}.

\begin{defn}\label{defn121124a}
Let $A,N$ be $R$-complexes.
\begin{enumerate}[(a)]
\item\label{defn121124a1}
$A$ is \emph{semidualizing} if it is homologically finite and the natural homothety morphism
$\chi^R_A\colon R\to\Rhom AA$ is an isomorphism in $\catd(R)$. 
\item
$A$ is \emph{tilting} if it is semidualizing and has finite projective dimension.
\item\label{defn121124a2}
$N$ is \emph{derived $A$-reflexive} if $N$ and $\Rhom NA$ are homologically finite and the natural biduality morphism
$\delta^A_N\colon N\to\Rhom {\Rhom NA}A$ is an isomorphism in $\catd(R)$.
\item\label{defn121124a3}
$N$ is in the \emph{Bass class} $\catb_A(R)$ if $N$ and $\Rhom AN$ are homologically bounded and the natural
evaluation morphism
$\xi^A_N\colon\Lotimes{A}{\Rhom AN}\to N$ is an isomorphism in $\catd(R)$.
\item\label{defn121124a4}
$N$ is in the \emph{Auslander class} $\cata_A(R)$ if $N$ and $\Lotimes AN$ are homologically bounded and the natural
 morphism
$\gamma^A_N\colon N\to \Rhom{A}{\Lotimes AN}$ is an isomorphism in $\catd(R)$.
\end{enumerate}
\end{defn}

The following facts are straightforward to verify.

\begin{fact}\label{fact121124b}
Let $A$ be a finitely generated $R$-module, and let $N$ be an $R$-module.
\begin{enumerate}[(a)]
\item\label{fact121124b1}
$A$ is semidualizing as an $R$-complex if and only if it satisfies the conditions in Fact~\ref{fact121124a}\eqref{fact121124b1}.
\item
$A$ is tilting as an $R$-complex if and only if it is a rank-1 projective $R$-module.
\item\label{fact121124b2}
An $R$-module that is totally $A$-reflexive (as in Fact~\ref{fact121124a}\eqref{fact121124b2}) is derived $A$-reflexive.
If $N$ has a finite resolution by totally $A$-reflexive $R$-modules, then it is derived $A$-reflexive; the converse holds when $A$ is semidualizing as in~\cite{yassemi:gd}.
\item\label{fact121124b3}
If the natural map $\Otimes{A}{\Hom AN}\to N$ is bijective and $\Ext iAN=0=\Tor iA{\Hom AN}$ for all $i\geq 1$, then
$N\in\catb_A(R)$; the converse holds when $A$ is semidualizing by~\cite[(4.10) Observation]{christensen:scatac}.
\item\label{fact121124b4}
If the natural map $N\to \Hom{A}{\Otimes AN}$ is bijective and $\Tor iAN=0=\Ext iA{\Otimes AN}$ for all $i\geq 1$, then
$N\in\cata_A(R)$; the converse holds when $A$ is semidualizing by~\cite[(4.10) Observation]{christensen:scatac}.
\end{enumerate}
\end{fact}

\begin{lem}\label{lem121125a}
Assume that $R$ is local, and let $A$ and $B$ be  homologically finite $R$-complexes such that $A\not\simeq 0$.
Then the following conditions are equivalent:
\begin{enumerate}[\rm(i)]
\item \label{lem121125a1}
$B\simeq 0$,
\item \label{lem121125a2}
$\Lotimes AB\simeq 0$, 
\item \label{lem121125a3}
$\Rhom AB\simeq 0$, and
\item \label{lem121125a4}
$\Rhom BA\simeq 0$.
\end{enumerate}
\end{lem}

\begin{proof}
For $n=$ii,iii,iv, the implications \eqref{lem121125a1}$\implies(n)$ are standard.
For the converses, we suppose that $B\not\simeq 0$, and
conclude that $\Lotimes AB\not\simeq 0$, 
$\Rhom AB\not\simeq 0$, and
$\Rhom BA\not\simeq 0$.
For instance, this follows from the Bass series and Poincar\'e series computations in~\cite[Lemma (1.5.3)]{avramov:rhafgd}.
\end{proof}

\begin{lem}\label{lem121125b}
Assume that $R$ is local, and let $A$, $X$, and $Y$ be  homologically finite $R$-complexes such that $A\not\simeq 0$.
Given a morphism $f\colon X\to Y$ the following conditions are equivalent:
\begin{enumerate}[\rm(i)]
\item \label{lem121125b1}
$f$ is an isomorphism in $D(R)$,
\item \label{lem121125b2}
$\Lotimes Af$ is an isomorphism in $D(R)$, 
\item \label{lem121125b3}
$\Rhom Af$ is an isomorphism in $D(R)$, and
\item \label{lem121125b4}
$\Rhom fA$ is an isomorphism in $D(R)$.
\end{enumerate}
\end{lem}

\begin{proof}
Apply Lemma~\ref{lem121125a} to the mapping cone $B:=\cone(f)$.
\end{proof}

\begin{fact}\label{prop121126a}
Let $A$ a homologically finite $R$-complex.
Then the following conditions are equivalent:
\begin{enumerate}[\rm(i)]
\item \label{prop121126a1}
$A$ is semidualizing over $R$,
\item \label{prop121126a2}
there is an isomorphism $\Rhom AA\simeq R$ in $D(R)$, 
\item \label{prop121126a3}
for each maximal ideal $\m\subset R$, there is an isomorphism $\Rhom[R_{\m}] {A_{\m}}{A_{\m}}\simeq R_{\m}$ in $D(R_{\m})$, 
\item \label{prop121126a4}
$R$ is derived $A$-reflexive,
\item \label{prop121126a5}
$A$ is derived $A$-reflexive and $\supp_R(A)=\spec(R)$, and
\item \label{prop121126a6}
$U^{-1}A$ is semidualizing over $U^{-1}R$ for each multiplicatively closed $U\subseteq R$.
\end{enumerate}
Indeed, in addition to~\cite[Proposition 3.1]{avramov:rrc1}, it suffices to note that the implications
$\eqref{prop121126a1}\implies\eqref{prop121126a2}\implies\eqref{prop121126a3}$ are straightforward.
\end{fact}

\begin{rmk}\label{rmk121129a}
Assume  that $R_1$ and $R_2$ are commutative noetherian rings such that $R\cong R_1\times R_2$.
Using the natural idempotents in $R$, one checks readily that every $R$-complex is isomorphic to one of the
form $X_1\times X_2$ where $X_i$ is an $R_i$-complex for $i=1,2$. 

For $i=1,2$ let $X_i$, $Y_i$, and $Z_i$ be $R_i$-complexes. Recall that there are natural isomorphisms in $\catd(R)$:
\begin{align*}
\Rhom{X_1\times X_2}{Y_1\times Y_2}
&\simeq\Rhom[R_1]{X_1}{Y_1}\times\Rhom[R_2]{X_2}{Y_2}
\\
\Lotimes{(X_1\times X_2)}{(Y_1\times Y_2)}
&\simeq(\Lotimes[R_1]{X_1}{Y_1})\times(\Lotimes[R_2]{X_2}{Y_2}).
\end{align*}
From this, it follows that
\begin{enumerate}[\rm(a)]
\item\label{rmk121129a1}
$X_1\times X_2$ is semidualizing for $R$ if and only if each $X_i$ is semidualizing for $R_i$.
\item\label{rmk121129a2}
$\Rhom{X_1\times X_2}{Y_1\times Y_2}$ is semidualizing for $R$ if and only if $\Rhom{X_i}{Y_i}$ is semidualizing for $R_i$ for $i=1,2$.
\item\label{rmk121129a3}
$X_1\times X_2$ is derived $Y_1\times Y_2$-reflexive if and only if each $X_i$ is derived $Y_i$-reflexive.
\item\label{rmk121129a4}
$\Rhom{X_1\times X_2}{Y_1\times Y_2}$ is derived $Z_1\times Z_2$-reflexive if and only if 
the complex $\Rhom{X_i}{Y_i}$ is derived $Z_i$-reflexive for $i=1,2$.
\item\label{rmk121129a5}
$X_1\times X_2\in\catb_{Y_1\times Y_2}(R)$ if and only if $X_i\in\catb_{Y_i}(R_i)$ for $i=1,2$.
\item\label{rmk121129a6}
$X_1\times X_2\in\cata_{Y_1\times Y_2}(R)$ if and only if $X_i\in\cata_{Y_i}(R_i)$ for $i=1,2$.
\end{enumerate}
\end{rmk}

\begin{defn}\label{defn121129a}
The \emph{semidualizing locus} of a homologically finite $R$-complex $A$ is
$$\sd_R(A):=\{\p\in\spec(R)\mid\text{$A_{\p}$ is semidualizing for $R_{\p}$}\}.$$
\end{defn}

\begin{rmk}\label{rmk121129b}
Let $A$ be a homologically finite $R$-complex.
Then we have 
$$\spec(R)\ssm\supp_R(\cone(\chi^R_A))=\sd_R(A)\subseteq\supp_R(A).$$
Also, $A$ is semidualizing for $R$ if and only if $\sd_R(A)=\spec(R)$; see Fact~\ref{prop121126a}.
\end{rmk}

\begin{lem}\label{lem121129a}
Let $A$ be a homologically finite $R$-complex such that $\Rhom AA$ is homologically finite,
i.e., such that $\Ext iAA=0$ for $i\gg 0$.
Then $\sd_R(A)$ is Zariski open in $\spec(R)$. 
\end{lem}

\begin{proof}
As $\Rhom AA$ is homologically finite, so is the mapping cone $\cone(\chi^R_A)$.
So, the set $\sd_R(A)=\spec(R)\ssm\supp_R(\cone(\chi^R_C))$ is open in $\spec(R)$.
\end{proof}

\section{Results}\label{sec121130b}

We begin this section with the local version of our main results.

\begin{thm}\label{thm121124e}
Assume that $R$ is local, and let $A$ be a homologically finite $R$-complex.
Then the following conditions are equivalent:
\begin{enumerate}[\rm(i)]
\item \label{thm121124e1}
$A$ is semidualizing for $R$,
\item \label{thm121124e2}
$\Rhom AA$ is semidualizing for $R$,
\item \label{thm121124e3}
$A$ is derived $A$-reflexive and $A\not\simeq 0$,
\item \label{thm121124e4}
$\Rhom AA$ is derived $A$-reflexive and $A\not\simeq 0$,
\item \label{thm121124e5}
$A\in\catb_A(R)$ and $A\not\simeq 0$, and
\item \label{thm121124e6}
$R\in\cata_A(R)$.
\end{enumerate}
\end{thm}

\begin{proof}
Note that if $A$ is semidualizing for $R$, then $A\not\simeq 0$ since $0\simeq\Rhom 00\not\simeq R$.
Similarly, if $\Rhom AA$ is semidualizing for $R$, then $A\not\simeq 0$.
Thus, for $n=$ii,iii,iv,v, the implications
$\eqref{thm121124e1}\implies (n)$
are from~\cite[Theorem 1.3]{frankild:rbsc}.

$\eqref{thm121124e2}\implies\eqref{thm121124e1}$
Assume that $\Rhom AA$ is semidualizing for $R$, and consider the following commutative diagram in $\catd(R)$.
$$\xymatrix@C=15mm{
R\ar[r]^-{\chi^R_{\Rhom AA}}_-{\simeq}\ar[d]_{\chi^R_A}
&\Rhom{\Rhom AA}{\Rhom AA}\ar[d]^{\simeq} \\
\Rhom AA\ar[r]^-{\Rhom{\xi^A_A}A}
&\Rhom{\Lotimes A{\Rhom AA}}A}$$
The unspecified isomorphism is Hom-tensor adjointness. 
From this, it follows that there is a monomorphism $R\into \HH_0(\Rhom AA)$, so 
$\HH_0(\Rhom AA)\neq 0$. 
From this, we conclude that a minimal free resolution $F\simeq\Rhom AA$ has $F_0\neq 0$.
Thus, there is a coefficient-wise inequality $P^R_{\Rhom AA}(t)\succeq 1$.

From the above diagram, it follows that the composition $\Rhom{\xi^A_A}A\circ \chi^R_A$ is an isomorphism, hence, so is the induced morphism
$$\Rhom k{\Rhom{\xi^A_A}A\circ\chi^R_A}=
\Rhom k{\Rhom{\xi^A_A}A}\circ\Rhom k{\chi^R_A}$$
where $k$ is the residue field of $R$.
In particular, the induced map on homology
$$\Ext ikR\to\Ext ik{\Rhom AA}$$
is a monomorphism for each $i$.
This explains the first coefficient-wise inequality in the next sequence:
\begin{align*}
I_R^{\Rhom AA}(t)
&\succeq I^R_R(t)
=P^R_{\Rhom AA}(t)I_R^{\Rhom AA}(t)
\succeq I_R^{\Rhom AA}(t).
\end{align*}
The equality follows from the fact that $\Rhom AA$ is semidualizing, by~\cite[1.5]{frankild:rrhffd}.
The second coefficient-wise inequality is from the condition $P^R_{\Rhom AA}(t)\succeq 1$ established above.
It follows that we have a coefficient-wise equality
$$P^R_{\Rhom AA}(t)I_R^{\Rhom AA}(t)
= I_R^{\Rhom AA}(t)$$
From this, we conclude that $P^R_{\Rhom AA}(t)=1$, so $\Rhom AA\simeq R$ and $A$ is semidualizing by Fact~\ref{prop121126a}.

$\eqref{thm121124e3}\implies\eqref{thm121124e1}$
Assume that $A$ is derived $A$-reflexive and $A\not\simeq 0$.
It follows that $\Rhom AA$ is homologically finite.
Consider the following commutative diagram in $\catd(R)$ where the unspecified isomorphism is Hom-cancellation.
$$\xymatrix{
A\ar[r]^-{\delta^A_A}_-{\simeq}\ar[d]_=
&\Rhom{\Rhom AA}A\ar[d]^{\Rhom{\chi^R_A}A} \\
A&\Rhom RA\ar[l]^-\simeq}$$
It follows that $\Rhom{\chi^R_A}A$ is an isomorphism in $\catd(R)$, so Lemma~\ref{lem121125b}
implies that $\chi^R_A$ is an isomorphism in $\catd(R)$, thus $A$ is semidualizing.

$\eqref{thm121124e4}\implies\eqref{thm121124e3}$
Assume that $\Rhom AA$ is derived $A$-reflexive and $A\not\simeq 0$. 
In particular, the biduality morphism 
$$\delta^A_{\Rhom AA}\colon {\Rhom AA}\to\Rhom{\Rhom {\Rhom AA}A}A$$ 
is an isomorphism in $\catd(R)$.
From~\cite[2.2]{avramov:rrc1} we conclude that $\Rhom{\delta^A_A}A$ is an isomorphism in $\catd(R)$.
Since $A$ and $\Rhom{\Rhom AA}{A}$ are both homologically finite by assumption, Lemma~\ref{lem121125b}
implies that $\delta^A_A$ is an isomorphism in $\catd(R)$.
It follows that $A$ is derived $A$-reflexive.

$\eqref{thm121124e5}\implies\eqref{thm121124e1}$
Assume that $A\in\catb_A(R)$, and consider the commutative diagram
$$\xymatrix@C=12mm{
\Lotimes AR\ar[r]^-{\Lotimes A{\chi_A^A}}
\ar[d]_\simeq
&\Lotimes A{\Rhom AA}\ar[ld]^{\xi^A_A}\\
A}$$
As in the previous paragraphs, Lemma~\ref{lem121125b} implies that $A$ is semidualizing.

$\eqref{thm121124e1}\iff\eqref{thm121124e6}$
This follows readily from the next commutative diagram in $\catd(R)$.
$$\xymatrix{R\ar[r]^-{\xi_R^A}\ar[d]_-{\chi^R_A}&\Rhom{A}{\Lotimes AR}\ar[ld]^-{\simeq} \\
\Rhom AA}$$
See also~\cite[(4.4) Proposition]{christensen:scatac} for one implication.
\end{proof}

\begin{rmk}\label{rmk121128a}
In Theorem~\ref{thm121124e}
the implications
$\eqref{thm121124e6}\implies\eqref{thm121124e1}\implies (n)$
for $n=$ii,iii,iv,v,vi do not use the local assumption. 
The point of much of the remainder of this paper is that the implications
$(n)\implies\eqref{thm121124e1}$ fail in general for $n=$ii,iii,iv,v. Moreover, we explicitly characterize the failure of these implications. 
\end{rmk}

The proof of the next result is similar to the proof of~\cite[Theorem 3.2]{frankild:sdcms}.

\begin{thm}\label{thm121127a}
Assume that $R$ is local, and
let $A$ be a homologically finite $R$-complex.
Then $0\not\simeq A\in\cata_A(R)$ if and only if $A\sim R$.
\end{thm}

\begin{proof}
One implication is straightforward: if $A\sim R$, then $\cata_A(R)$ contains all homologically bounded complexes, so $A\in\cata_A(R)$ and $0\not\simeq R\sim A$.

For the converse, assume that $0\not\simeq A\in\cata_A(R)$. Shift $A$ if necessary to assume that 
$\inf\{n\in\bbz\mid\HH_n(A)\neq 0\}=0$.
Let $P\simeq A$ be a minimal free resolution of $A$. 
It follows that $P_i=0$ for all $i<0$ and $P_0\neq 0$.
The condition $P\simeq A\in \cata_A(R)$ implies that the natural map $\gamma^P_P\colon P\to\Hom P{\Otimes PP}$ is a quasiisomorphism,
hence it induces the quasiisomorphism in the top row of the next commutative diagram of chain maps
where the unspecified isomorphism is Hom-tensor adjointness.
$$\xymatrix@C=19mm{
\Hom PP \ar[r]^-{\Hom P{\gamma^P_P}}_-{\simeq}\ar[rd]_{\Delta}
&\Hom P{\Hom P{\Otimes PP}}\ar[d]^{\cong} \\
&\Hom{\Otimes PP}{\Otimes PP}}$$
In degree 0, the composition $\Delta$ is given by $f\mapsto \Otimes Pf$.
The diagram shows that $\Delta$ is a quasiisomorphism.

Given two $R$-complexes $X$ and $Y$, let $\Theta_{X,Y}\colon\Otimes XY\to\Otimes YX$ by the natural commutativity isomorphism
$\Otimes[] xy\mapsto(-1)^{|x||y|}\Otimes[] yx$. This is a chain map, hence the fact that $\Delta$ is a quasiisomorphism implies that there is a
chain map $f\colon P\to P$ such that $\Otimes Pf\colon \Otimes PP\to\Otimes PP$ is homotopic to $\Theta_{P,P}$.

Let $k$ be the residue field of $R$, and set $\ol{(-)}=\Otimes k-$. The previous paragraph implies that 
$\ol{\Otimes Pf}\colon \ol{\Otimes PP}\to\ol{\Otimes PP}$ is homotopic to $\ol{\Theta_{P,P}}$.
Using the natural isomorphism $\ol{\Otimes --}\cong\Otimes[k]{\ol -}{\ol -}$, it follows that
$\Otimes[k] {\ol{P}}{\ol{f}}\colon \Otimes[k] {\ol{P}}{\ol{P}}\to\Otimes[k] {\ol{P}}{\ol{P}}$ is homotopic to $\Theta_{\ol{P},\ol{P}}$.
Since $P$ is minimal, the differentials on $\ol P$ and $\Otimes[k] {\ol{P}}{\ol{P}}$ are 0, and it follows that
$\Otimes[k] {\ol{P}}{\ol{f}}=\Theta_{\ol{P},\ol{P}}\colon \Otimes[k] {\ol{P}}{\ol{P}}\to\Otimes[k] {\ol{P}}{\ol{P}}$.

We first show that $\ol{P_0}\cong k$. Since $\ol{P_0}$ is a non-zero $k$-vector space, it suffices to show that
$\rank_k(\ol{P_0})\leq 1$. Suppose that $r=\rank_k(\ol{P_0})\geq 2$, and let $x_1,\ldots,x_r\in\ol{P_0}$ be a basis.
It follows that $\Otimes[k]{\ol{P_0}}{\ol{P_0}}$ has rank $r^2$ with basis $\{\Otimes[]{x_i}{x_j}\mid i,j=1,\ldots,r\}$.
The equality $\Otimes[k] {\ol{P}}{\ol{f}}=\Theta_{\ol{P},\ol{P}}$ implies that
$$\Otimes[]{x_2}{x_1}=\Otimes[]{x_1}{f(x_2)}\in\operatorname{Span}_k\{\Otimes[]{x_1}{x_1},\ldots,\Otimes[]{x_1}{x_r}\}$$
contradicting the linear independence of the given basis for $\Otimes[k]{\ol{P_0}}{\ol{P_0}}$.

We now show that $\ol{P_i}=0$ for all $i\neq 0$. (It then follows that $A\simeq P\cong R$, as desired.)
Let $i\geq 1$ and $y\in\ol{P_i}$. With $x_1$ as in the previous paragraph, 
the equality $\Otimes[k] {\ol{P}}{\ol{f}}=\Theta_{\ol{P},\ol{P}}$ implies that
$$0=\Otimes[]{y}{x_1}-\Otimes[]{x_1}{f(y)}\in(\Otimes[k]{\ol{P_i}}{\ol{P_0}})\oplus(\Otimes[k]{\ol{P_0}}{\ol{P_i}}).$$
Since $i\neq 0$, we have $(\Otimes[k]{\ol{P_0}}{\ol{P_i}})\cap(\Otimes[k]{\ol{P_i}}{\ol{P_0}})=0$,
so we conclude that $\Otimes[]{y}{x_1}=0$ in $\Otimes[k]{\ol{P_i}}{\ol{P_0}}$.
Since $0\neq x_1$ in the vector space $\ol{P_0}$, it follows that $y=0$.
The element $y\in\ol{P_i}$ was chosen arbitrarily, so we conclude that 
$\ol{P_i}=0$, as desired.
\end{proof}

Next, we present our non-local results.

\begin{cor}\label{thm121124b'}
Let $A$ be a homologically finite $R$-complex.
Then the following conditions are equivalent:
\begin{enumerate}[\rm(i)]
\item \label{thm121124b'1}
$A$ is semidualizing for $R$,
\item \label{thm121124b'2}
$\Rhom AA$ is semidualizing for $R$, 
\item \label{thm121124b'3}
$A$ is derived $A$-reflexive and $\supp_R(A)=\spec(R)$,
\item \label{thm121124b'4}
$\Rhom AA$ is derived $A$-reflexive and $\supp_R(A)=\spec(R)$,
\item \label{thm121124b'5}
$A\in\catb_A(R)$ and $\supp_R(A)=\spec(R)$, and
\item \label{thm121124b'6}
$R\in\cata_A(R)$.
\end{enumerate}
\end{cor}

\begin{proof}
Note that conditions \eqref{thm121124b'1}, \eqref{thm121124b'2}, and \eqref{thm121124b'6} all imply that $\supp_R(A)=\spec(R)$
since $A$ is homologically finite.
The implications $\eqref{thm121124b'6}\iff\eqref{thm121124b'1}\implies(n)$ for $n=$ii,iii,iv,v follow from Remark~\ref{rmk121128a}.
For the implications $(n)\implies\eqref{thm121124b'1}$  with $n=$ii,iii,iv,v, note that condition $(n)$ localizes; 
since the semidualizing property is local by Fact~\ref{prop121126a},
the desired conclusion follows from Theorem~\ref{thm121124e}.
\end{proof}

The next result is proved like the previous one, via Theorem~\ref{thm121127a}.

\begin{cor}\label{thm121129a}
Let $A$ be a homologically finite $R$-complex.
Then $A\in\cata_A(R)$ and $\supp_R(A)=\spec(R)$ if and only if $A$ is a tilting $R$-complex.
\end{cor}

As we show in~\ref{proof121130a} below, the next result is the key to proving
Theorem~\ref{thm121124a}.

\begin{thm}\label{thm121124b}
Let $A$ be a homologically finite $R$-complex.
Then the following conditions are equivalent:
\begin{enumerate}[\rm(i)]
\item \label{thm121124b1}
There are non-zero commutative noetherian rings $R_1,R_2$ with identity
such that $R\cong R_1\times R_2$, and there is a semidualizing $R_1$-complex $A_1$
such that $A\cong A_1\times 0$,
\item \label{thm121124b3}
$A$ is derived $A$-reflexive and not semidualizing such that $A\not\simeq 0$,
\item \label{thm121124b4}
$\Rhom AA$ is derived $A$-reflexive, $A$ is not semidualizing, and $A\not\simeq 0$, and
\item \label{thm121124b5}
$0\not\simeq A\in\catb_A(R)$ and $A$ is not semidualizing.
\end{enumerate}
In particular, when the above conditions are satisfied, $\spec(R)$ is disconnected.
\end{thm}

\begin{proof}
$\eqref{thm121124b1}\implies\eqref{thm121124b3}$
Assume that there are non-zero commutative noetherian rings $R_1,R_2$ with identity
such that $R\cong R_1\times R_2$, and  that there is a semidualizing $R_1$-complex $A_1$
such that $A\cong A_1\times 0$.
Since $R_2\neq 0$, we conclude that $0$ is not semidualizing for $R_2$, so Remark~\ref{rmk121129a}\eqref{rmk121129a1} implies
that $A$ is not semidualizing for $R$.
Since $A_1$ is semidualizing for $R_1\neq 0$, we conclude that $A\not\simeq 0$, and that
$A$ is derived $A$-reflexive by Remarks~\ref{rmk121129a}\eqref{rmk121129a3} and~\ref{rmk121128a}.

$\eqref{thm121124b3}\implies\eqref{thm121124b1}$
Assume that $A$ is derived $A$-reflexive and not semidualizing such that $A\not\simeq 0$.
In particular, the complex $\Rhom AA$ is homologically finite.
Lemma~\ref{lem121129a} implies that $\sd_R(A)$ is an open subset of $\spec(R)$.

We claim that $\sd_R(A)=\supp_R(A)$. One containment is from Remark~\ref{rmk121129b}.
For the reverse containment, let $\p\in\supp_R(A)$. It follows that $A_{\p}\not\simeq 0$ is totally $A_{\p}$-reflexive,
so Theorem~\ref{thm121124e} implies that $A_{\p}$ is semidualizing for $R_{\p}$, i.e., $\p\in\sd_R(A)$.

It follows that $\sd_R(A)=\supp_R(A)$ is both open and closed in $\spec(R)$. 
Since $A$ is not semidualizing, Remark~\ref{rmk121129b} shows that $\sd_R(A)=\supp_R(A)\neq\spec(R)$.
On the other hand, since $A\not\simeq 0$, we have $\sd_R(A)=\supp_R(A)\neq\emptyset$.
It follows that $\spec(R)=\supp_R(A)\uplus(\spec(R)\ssm\supp_R(A))$ is a disconnection of $\spec(R)$.
A standard result implies that there are  commutative rings $R_1$ and $R_2$ such that
\begin{enumerate}[(1)]
\item\label{item111229a}
$R\cong R_1\times R_2$, and
\item\label{item111229b}
under the natural bijection $\spec(R)\cong\spec(R_1)\uplus\spec(R_2)$, the set $\supp_R(A)$ corresponds to $\spec(R_1)$,
and $\spec(R)\ssm\supp_R(A)$ corresponds to $\spec(R_2)$.
\end{enumerate}
Remark~\ref{rmk121129a} implies that for $i=1,2$ there is an $R_i$-complex $A_i$ such that $A\simeq A_1\times A_2$.
Under the natural bijection $\spec(R)\cong\spec(R_1)\uplus\spec(R_2)$, for each $P\in\spec(R)$ and its corresponding prime
$\p_i\in\spec(R_i)$, we have $A_P\simeq (A_i)_{\p_i}$. Using condition~\eqref{item111229b} above, 
it follows that
\begin{enumerate}[(3)]
\item[(3)]
for each $\p_1\in\spec(R_1)$, corresponding to $P\in\supp_R(A)=\sd_R(A)$, since $A_P$ is semidualizing for $R_P$,
the complex $(A_1)_{\p_1}$ is semidualizing for $(R_1)_{\p_1}$, and
\item[(4)]
for each $\p_2\in\spec(R_2)$ corresponding to $P\in\spec(R)\ssm\supp_R(A)$, we have $(A_2)_{\p_2}\simeq A_P\simeq 0$.
\end{enumerate}
Because of condition~(3), Fact~\ref{prop121126a} implies that $A_1$ is semidualizing for $R_1$.
And condition~(4) implies that $\supp_{R_2}(A_2)=\emptyset$, so $A_2\simeq 0$, as desired.

For $n=$iii,iv, the equivalence $\eqref{thm121124b1}\iff(n)$ is proved similarly.
\end{proof}

\begin{thm}\label{thm121124c}
Let $A$ be a homologically finite $R$-complex.
Then the following conditions are equivalent:
\begin{enumerate}[\rm(i)]
\item \label{thm121124c1}
$0\not\simeq A\in\cata_A(R)$ and $A$ is not semidualizing for $R$, and
\item \label{thm121124c2}
there are non-zero commutative noetherian rings $R_1,R_2$ with identity
such that $R\cong R_1\times R_2$, and there is a tilting $R_1$-complex $A_1$
such that $A\cong A_1\times 0$.
\end{enumerate}
\end{thm}

\begin{proof}
From~\cite[Proposition 4.4]{frankild:rbsc}, we know that $A$ is tilting if and only if $A_{\m}\sim R_{\m}$
for each maximal ideal $\m\subset R$. Thus, the desired implications 
follow from Theorem~\ref{thm121127a}
as in the proof of Theorem~\ref{thm121124b}.
\end{proof}

\begin{para}[Proof of Theorem~\ref{thm121124a}] \label{proof121130a}
Let $A$ be a non-zero totally $A$-reflexive $R$-module that is not semidualizing.
Then $A$ is derived $A$-reflexive and not semidualizing such that $A\not\simeq 0$,
so the desired conclusion follows from Theorem~\ref{thm121124b}. This uses the fact that if $A\simeq A_1\times 0$,
then $A_1$ is isomorphic in $\catd(R)$ to a module and $A\cong A_1\times 0$.
\qed
\end{para}

\begin{rmk}\label{rmk121201a}
Other results for modules can be deduced from our results for complexes.
We leave it as an exercise for the interested reader to formulate them.
\end{rmk}

We end with two consequences for integral domains that parallel our local results.

\begin{cor}\label{cor121130a}
Assume that $R$ is an integral domain, and let $A$ be a homologically finite $R$-complex.
Then the following conditions are equivalent:
\begin{enumerate}[\rm(i)]
\item \label{cor121130a1}
$A$ is a semidualizing $R$-complex,
\item \label{cor121130a2}
$A$ is derived $A$-reflexive and $A\not\simeq 0$,
\item \label{cor121130a3}
$\Rhom AA$ is derived $A$-reflexive and $A\not\simeq 0$, and
\item \label{cor121130a4}
$0\not\simeq A\in\catb_A(R)$.
\end{enumerate}
\end{cor}

\begin{proof}
$\eqref{cor121130a2}\implies\eqref{cor121130a1}$
Assume that $A$ is derived $A$-reflexive and $A\not\simeq 0$.
If $A$ is not semidualizing, then 
Theorem~\ref{thm121124b} provides  a non-trivial decomposition $R\cong R_1\times R_2$, contradicting the assumption that $R$ is a domain.

The remaining implications follow similarly, using Remark~\ref{rmk121128a}.
\end{proof}

The next result is proved like the previous one, using Theorem~\ref{thm121124c}.

\begin{cor}\label{cor121130b}
Assume that $R$ is an integral domain, and let $A$ be a homologically finite $R$-complex.
Then 
$0\not\simeq A\in\cata_A(R)$
if and only if
$A$ is a tilting $R$-complex.
\end{cor}

\section*{Acknowledgments}
We thank Jon Totushek for posing the motivating question for this work.

\providecommand{\bysame}{\leavevmode\hbox to3em{\hrulefill}\thinspace}
\providecommand{\MR}{\relax\ifhmode\unskip\space\fi MR }
\providecommand{\MRhref}[2]{%
  \href{http://www.ams.org/mathscinet-getitem?mr=#1}{#2}
}
\providecommand{\href}[2]{#2}

\end{document}